\documentclass[11pt]{article}

\usepackage{amsmath, amsfonts,amssymb, amsthm}
\usepackage{graphicx}
\usepackage{color}
\usepackage{epsf}
\usepackage{tikz}
\usepackage{subfig}
\usepackage{hyperref}
\hypersetup{
    colorlinks = true,
    linkcolor = {black},
    citecolor = {black}
}

\usepackage[left=1.25in, right=1.25in, top=1.1in, bottom=1.25in]{geometry}

\usepackage{enumitem}
\setlist{itemsep=0pt, topsep=0pt}

\newtheorem{theorem}{Theorem}[section]
\newtheorem{corollary}[theorem]{Corollary}
\newtheorem{lemma}[theorem]{Lemma}

\newtheorem{claim}[theorem]{Claim}
\newtheorem{proposition}[theorem]{Proposition}
\newtheorem{observation}[theorem]{Observation}
\newtheorem{fact}[theorem]{Fact}

\newtheorem{question}[theorem]{Question}
\newtheorem{definition}[theorem]{Definition}

\newtheorem{remark}[theorem]{Remark}
\newcommand{\ep}{\epsilon}

\newcommand{\tbf}[1]{\textbf{#1}}

\newcommand{\set}[2]{\left\{#1\mathrel{}\middle|\mathrel{}#2\right\}}

\newcommand{\floor}[1]{\lfloor#1\rfloor}
\newcommand{\ceiling}[1]{\lceil#1\rceil}

\newcommand{\sqbs}[1]{\left[ #1 \right]}
\newcommand{\of}[1]{\left( #1 \right)}

\newcommand{\alf}{3.75}
\newcommand{\alff}{2.75}
\newcommand{\eps}{\varepsilon}

\title{New lower bounds on the size-Ramsey number of a path} 
\author{Deepak Bal$^{1}$ \and Louis DeBiasio$^{2}$}
\date{\today}

\begin{document}

\maketitle
\noindent\footnotetext[1]{Department of Mathematics, Montclair State University {\tt deepak.bal@montclair.edu}. }
\noindent\footnotetext[2]{Department of Mathematics, Miami University {\tt debiasld@miamioh.edu}. 
}

\begin{abstract}
We prove that for all graphs with at most $(3.75-o(1))n$ edges there exists a 2-coloring of the edges such that every monochromatic path has order less than $n$.  This was previously known to be true for graphs with at most $2.5n-7.5$ edges. We also improve on the best-known lower bounds in the $r$-color case.
\end{abstract}

\section{Introduction}
Given a graph $H$, let $\hat{R}_r(H)$ be the minimum $m$ such that there exists a graph $G$ with $m$ edges such that in every $r$-coloring of $G$, there is a monochromatic copy of $H$.  When $r=2$, we drop the subscript and just write $\hat{R}(H)$.  We refer to $\hat{R}(H)$ as the \emph{size-Ramsey} number of $H$.

Let $P_n$ be the path with $n$ vertices. Erd\H{o}s \cite{Er} famously asked if $\hat{R}(P_n) / n \to \infty$ and $\hat{R}(P_n) / n^2 \to 0$. Beck \cite{B2} proved that, in fact, $\hat{R}(P_n) \le 900n$ (for $n$ sufficiently large). The bound $900n$ was subsequently improved in \cite{B}, \cite{Bol}, \cite{DP1}, \cite{Let} and currently rests at $74n$ as proved by Dudek and Pra{\l}at in \cite{DP2}. 

As for the lower bound, it is clear that $\hat{R}(P_n)> 2n-4$ since $P_n$ has $n-1$ edges.  Beck \cite{B2} proved $\hat{R}(P_n)\geq (9/4-o(1))n$, Bielak \cite{Bie} proved $\hat{R}(P_n)\geq 9n/4-3$,  Bollob\'as \cite{B} proved $\hat{R}(P_n)\geq (1+\sqrt{2}-o(1))n$, and finally Dudek and Pra{\l}at \cite{DP2} proved $\hat{R}(P_n)\geq 5n/2-15/2$.

The closest thing there is to a conjecture about the precise value of $\hat{R}(P_n)$ is Bollob\'as' \cite{B} comment, ``it would not be surprising if $\hat{R}(P_n)$ turned out to be about $8n$.'' It is not known what insight led to this comment, but together with the recent flurry of activity on the upper bound, it inspired us to make a determined effort to improve the lower bound.  We prove the following.

\begin{theorem}\label{thm:main-2-col}
For all $\ep>0$, there exists $n_0$ such that if $n\geq n_0$ and $G$ is a graph with at most $(\alf-\ep) n$ edges, there exists a 2-coloring of the edges of $G$ such that every monochromatic path has order less than $n$.  Thus $\hat{R}(P_n) \ge (\alf - o(1))n$.
\end{theorem}

For the general, $r$-color version of the problem, the best upper bound is due to Krivelevich \cite{K} who proved $\hat{R}_r(P_n)=O(r^2\log (r) n)$ (Dudek and Pra{\l}at \cite{DP3} later gave a different proof).  In fact, both \cite{K} and \cite{DP3} prove the stronger ``density version'' of the theorem: there exists a graph $G$ (a binomial random graph) with $|E(G)| = O(r^2\log (r)n)$ such that every subgraph of $G$ with at least $e(G)/r$ many edges contains a monochromatic path of order $n$ (A recent paper of Balogh, Dudek, and Li \cite{BDL} shows that the factor $r^2\log r$ cannot be improved for this stronger density version in the setting of random graphs).
 
As for the lower bound, Dudek and Pra{\l}at \cite{DP2}  proved that for any $r\ge 2$, $\hat{R}_r(P_n) \ge \frac{(r+3)r}{4}n - O(r^2)$ and then Krivelevich \cite{K} proved that for any $r\ge 3$ such that $r-2$ is a prime power, $\hat{R}_r(P_n) \ge (r-2)^2n - o(n)$. We improve on each of these results by proving the following.

\begin{theorem}\label{thm:main-r-col}
Let $r \ge 2$ and let $q$ be the largest prime power such that $q\leq r-1$.  Then $$\hat{R}_r(P_n)\ge \max\left\{ \of{\frac{(r-1)r}{2} + \alff-o(1)}n, (q^2-o(1))n\right\}.$$ 
\end{theorem}

Note that the prime number theorem guarantees that for any $\eps>0$ and $r$ sufficiently large, there is a prime between $(1-\eps)r$ and $r$, so for sufficiently large $r$, the second term in the maximum will dominate and we have $\hat{R}_r(P_n) \ge (r-1 - o_r(1))^2 n$. Determining whether $\hat{R}_r(P_n)=\Theta(r^2)n$ or not is perhaps the most interesting open problem regarding the size-Ramsey number of a path.

\subsection{Outline, Notation} 

Our improvement in the lower bound stems from two main ideas.   

1) If we can partition the graph $G$ into sets of order at most $n-1$ such that the number of edges crossing the partition is at most $n-2$, then we can color the edges inside the sets red and the edges between the sets blue so there are no monochromatic $P_n$'s.  This has some similarity to the problem of determining the bisection width of a graph\footnote{the \emph{bisection width} of a graph is the minimum number of edges crossing a balanced bipartition of the vertex set} -- in which case a result of Alon \cite[Proposition 3.1]{A} gives good bounds on the number of crossing edges in a balanced bipartition of graphs with bounded maximum degree and at most $2n-2$ vertices.  However, in our case, $G$ may not have bounded maximum degree, $G$ may have more than $2n-2$ vertices, and we don't necessarily want the partition to be balanced.  Nevertheless, with some extra work, we are able to use similar methods from the study of the bisection width problem (e.g.\ \cite{A}, \cite{KM}) in our setting. 

2) From the ordinary path Ramsey problem it is known that if $G$ has at most $\frac{3n}{2}-2$ vertices, then there exists a 2-coloring of $G$ such that every monochromatic path has order less than $n$.  We show that if $G$ has between roughly $3n/2$ and $5n/3$ vertices and few enough edges, then there exists a 2-coloring of $G$ such that every monochromatic path has order less than $n$.  This allows us to only consider graphs with at least $5n/3$ vertices.

In Section \ref{sec:lems} we prove a number of lemmas which we will use throughout the proof. We also show how some of these lemmas imply the previously known lower bounds on the size-Ramsey number of paths.
In Section \ref{sec:2-col} we prove Theorem \ref{thm:main-2-col}. 
In Section \ref{sec:r-col} we prove Theorem \ref{thm:main-r-col}. In Section \ref{sec:concl}, we list a few observations and approaches that may helpful in trying to improve the lower bounds we have provided.

Let $G=(V,E)$ be a graph.  For all $S\subseteq V$, we write $G-S$ to mean $G[V \setminus S]$; i.e.\ the graph induced by $V\setminus S$.  Given disjoint sets $A, B\subseteq V$, we write $[A,B]$ to mean the bipartite subgraph of $G$ with vertex set $A\cup B$ and edges with one endpoint in $A$ and the other in $B$.  We sometimes write $|G|$ to mean $|V|$.  For any other notation we defer to \cite{Die}. 
All logarithms are natural (base $e$) unless otherwise stated.  Throughout the paper, if we refer to an \emph{$r$-coloring} of $G$, we mean an $r$-coloring of the edges of $G$.

\section{Lemmas}\label{sec:lems}

When proving a lower bound on the $r$-color size-Ramsey number of $P_n$, we are given a graph $G=(V,E)$ and we must exhibit an $r$-coloring of the edges of $G$ so that $G$ has no monochromatic paths of order $n$.  It is often useful to break this into cases depending the number of vertices of $G$.  In Section \ref{sec:Nlower} we use the examples from the ordinary path Ramsey problem to determine a lower bound on $|V|$.  In Section \ref{sec:Nupper} we prove a general result which allows us, when proving a lower bound on $\hat{R}_r(P_n)$, to restrict our attention to graphs with minimum degree at least $r+1$, which in turn gives us an upper bound on $|V|$.  In Section \ref{sec:prune}, we prove a lemma which we use in the proof of Theorem \ref{thm:main-r-col}. In Section \ref{sec:mainlem}, we prove the main lemma of the paper needed for the proof of Theorem \ref{thm:main-2-col}.  Finally, in Section \ref{extend} we show how to deal with the case when $G$ has between roughly $3n/2$ and $5n/3$ vertices.  

\subsection{Examples from the ordinary path Ramsey problem}\label{sec:Nlower}

\begin{proposition}[Gerencs\'er, \ Gy\'arf\'as \cite{GG}]\label{3n/2}
If $G$ has at most $\frac{3n}{2}-2$ vertices, then there exists a 2-coloring of $G$ such that every monochromatic path has order less than $n$.
\end{proposition}

\begin{proof}
Partition $V(G)$ into two sets $X_1, X_2$ with $|X_1|\leq \frac{n}{2}-1$ and $|X_2|\leq n-1$.  Color all edges incident with $X_1$ red and all edges inside $X_2$ blue.  Clearly the longest blue path has order $n-1$. Any pair of consecutive vertices on a red path must contain at least one vertex of $X_1$. Thus the longest red path is of order at most $2|X_1|+1\leq n-1$.
\end{proof}

\begin{proposition}[Yongqi, Yuansheng, Feng, Bingxi \cite{YYFB}]\label{3n/2_r}
Let $r\geq 3$. If $G$ has at most $2(r-1)(\frac{n}{2}-1)=(r-1)(n-2)$ vertices, then there exists an $r$-coloring of $G$ such that every monochromatic path has order less than $n$.
\end{proposition}

\begin{proof}
Partition $V(G)$ into $2r-2$ sets $X_1, X_2, \ldots, X_{2r-2}$ each of order at most $\frac{n}{2}-1$. In the following, addition is modulo $2r-2$. For $i = 1, \ldots, r-1$, color with color $i$, the edges between $X_i$ and $X_{i+1},\ldots, X_{i+r-2}$ and the edges between $X_{i+r-1}$ and $X_{i+r},\ldots X_{i+2r-3}$. 
Use color $r$ for the edges between $X_i$ and $X_{i+r-1}$ for $i=1,\ldots r-1$. Color arbitrarily within the $X_i$'s. This coloring has no monochromatic  $P_n$ in color $i$ for $i=1,\ldots r-1$ for the same reason as in Proposition \ref{3n/2}. There is none in color $r$ since each component of color $r$ is of order less than $n$.
\end{proof}

\subsection{A reduction lemma}\label{sec:Nupper}

\begin{fact}\label{Nupper}
If $G=(V,E)$ is a graph with minimum degree at least $r+1$, then $|V|\leq \frac{2|E|}{r+1}$.
\end{fact}

The following lemma shows that in order to get a lower bound on the $r$-color size-Ramsey number of $P_{n}$, we can restrict our attention to graphs $G$ with minimum degree at least $r+1$, and consequently at most $\frac{2|E|}{r+1}$ vertices.  This generalizes an observation which is implicit in the proof of Beck's lower bound \cite{B2}. 

\begin{lemma}\label{mindegree}
Let $r$ and $n$ be positive integers with $n\geq r+4$. If every connected graph with at most $m$ edges and minimum degree at least $r+1$ (and consequently at most $2m/(r+1)$ vertices) has an $r$-coloring such that every monochromatic path has order less than $n-2$, then every graph with at most $m$ edges has an $r$-coloring such that every monochromatic path has order less than $n$.
\end{lemma}

\begin{proof}
Suppose that every connected graph with at most $m$ edges and minimum degree at least $r+1$ has an $r$-coloring such that every monochromatic path has order less than $n-2$.  Let $G$ be a graph with at most $m$ edges.  Let $S=\{v\in V(G): d(v)\leq r\}$.  We begin by describing how to color the edges of $G-S$ so that $G-S$ contains no monochromatic paths of order $n-2$.  

If $G-S$ has fewer than $n-2$ vertices, then coloring the edges of $G-S$ arbitrarily we have an $r$-coloring of $G-S$ with no monochromatic paths of order $n-2$.  So suppose $G-S$ has at least $n-2\geq r+2$ vertices.  Let $v$ be a vertex in $G-S$ and suppose that $v$ has exactly $r+1-t$ neighbors in $G-S$ for some positive $t$.  This means $v$ had at least $t$ neighbors in $S$, so by making $v$ adjacent to $t$ vertices in $G-S$ (each of which was previously a non-neighbor of $v$) we make $v$ have degree at least $r+1$ and the total number of edges is still at most $m$.  We repeat this process for each vertex in $G-S$ which has degree less than $r+1$, updating on each step.  We end up with a graph $H$ such that $G-S\subseteq H$, $H$ has at most $m$ edges, and $\delta(H)\geq r+1$.  For each connected component of $H$, color the edges according to the hypothesis so that there are no monochromatic paths of order $n-2$.  This implies that $G-S$ has no monochromatic paths of order $n-2$.  

Now let $u_1, \dots, u_s$ be an arbitrary ordering of the vertices of $S$.  Since $\Delta(G[S])\leq r$, we color the edges incident with $u_1$ so that every edge receives a different color.  Let $2\leq i\leq s$ and suppose that for all $1\leq j\leq i-1$, we have colored all edges incident with $u_j$ so that if $C_j$ is the set of colors used on edges in $[\{u_j\}, \{u_1, \dots, u_{j-1}\}]$, then every other edge incident with $u_j$ gets a distinct color from $[r]\setminus C_j$.  Note that the only edges incident with $u_i$ which have already been colored are those which have one endpoint in $\{u_1, \dots, u_{i-1}\}$; let $C_i$ be the set of colors used on such edges.  Since $d(u_i)\leq r$ we can color the remaining edges incident with $u_i$ with distinct colors from $[r]\setminus C_i$.  

We have now colored all of the edges incident with $S$ such that every monochromatic component consisting of edges incident with $S$ is a star with all of its leaves in $S$. So every monochromatic path which only uses edges from $G-S$ has order less than $n-2$ and every monochromatic path which only uses edges from $E(G[S])\cup [S, V(G)-S]$ has order at most 3.  If a monochromatic, say color 1, path uses an edge from $[S, V(G)-S]$, then since its endpoint in $S$ is not incident with any other edges of color 1, this edge must be a pendant edge of the path (of which there are only two) and thus the longest monochromatic path in $G$ has order less than $(n-2)+2=n$.
\end{proof}

\begin{corollary}\label{size-ordinary}
For all $n\geq r+4$,
$\hat{R}_r(P_n)\geq \frac{r+1}{2}\cdot R_r(P_{n-2})$.
\end{corollary}

\begin{proof}
Let $G=(V,E)$ be a connected graph with fewer than $\frac{r+1}{2}\cdot R_r(P_{n-2})$ edges and minimum degree at least $r+1$.  So $|V|\leq \frac{2|E|}{r+1}<R_r(P_{n-2})$ and thus $G$ has an $r$-coloring with no monochromatic $P_{n-2}$.  So by Lemma \ref{mindegree}, every graph with fewer than $\frac{r+1}{2}\cdot R_r(P_{n-2})$ edges has an $r$-coloring with no monochromatic $P_n$.  
\end{proof}

\begin{remark}
Proposition \ref{3n/2} and Corollary \ref{size-ordinary} imply that $$\hat{R}(P_n)\geq \frac{3}{2}\cdot R(P_{n-2})\geq \frac{3}{2}\left(\frac{3}{2}(n-2)-\frac{3}{2}\right)=\frac{9}{4}n-\frac{27}{4}.$$
\end{remark}

%\begin{proof}
%By Lemma \ref{mindegree} we may assume that $|V|\leq \frac{2|E|}{3}\leq \frac{3}{2}(n-2)-2$ and thus we are done by Proposition \ref{3n/2}.
%\end{proof}

\begin{remark}\label{r^2-1}
Proposition \ref{3n/2_r} and Corollary \ref{size-ordinary} imply that for $r\geq 3$, $$\hat{R}_r(P_n)\geq \frac{r+1}{2}\cdot R_r(P_{n-2})> \frac{r+1}{2}(r-1)(n-4)=\frac{r^2-1}{2}n-2(r^2-1).$$
\end{remark}

%\begin{proof}
%By Lemma \ref{mindegree} we may assume that $|V|\leq \frac{2|E|}{r+1}\leq (r-1)(n-4)$ and thus we are done by Proposition \ref{3n/2_r}.
%\end{proof}

\begin{remark}
The bound in Remark \ref{r^2-1} is less than the bounds given in Theorem \ref{thm:main-r-col}.  However, Remark \ref{r^2-1} is the easiest way to see that $\hat{R}_r(P_n)=\Omega(r^2n)$.
\end{remark}

\subsection{Trimming a tree so that no long paths remain}\label{sec:prune}

The following is a slight generalization of the lemma used in \cite{B} and \cite{DP2} to give a lower bound on the size-Ramsey number of a path.

\begin{lemma}\label{snip}
For every tree $T$ with $|V(T)|\geq \floor{n/2}$, there exists a set $E'$ of at most $\floor{\frac{|V|}{\floor{n/2}}}-1$ edges such that $T-E'$ has no paths of order $n$.  
%In particular, if $T$ has at most $4n$ vertices, then $8$ edges suffice.  
\end{lemma}

\begin{proof}
If $T$ has no path of order $n$ we are done, so choose a path of order $n$ and delete the middle edge (or one of the two middle edges if $n$ is odd).  This separates $T$ into two subtrees, each with at least $\floor{n/2}$ vertices.  Now repeat on each subtree and call the set of deleted edges, $E'$.  When the process stops, every component of $T-E'$ has at least  $\floor{n/2}$ vertices and no paths of order $n$.
 Thus $T-E'$ has at most $\floor{\frac{|V|}{\floor{n/2}}}$ components, which means  $|E'| \le \floor{\frac{|V|}{\floor{n/2}}}-1$.
\end{proof}

\begin{remark}
Proposition \ref{3n/2} and Lemma \ref{snip} imply that $\hat{R}(P_n)\geq \frac{5}{2}n-7$.
\end{remark}

\begin{proof}
Let $G=(V,E)$ be a graph with at most $\frac{5n}{2}-\frac{15}{2}$ edges.  We may assume $G$ is connected and by Proposition \ref{3n/2} we may assume $\frac{3n}{2}-\frac{3}{2}\leq |V|$.  Let $T$ be a spanning tree of $G$ (which contains at least $\frac{3n}{2}-\frac{5}{2}$ edges).  Applying Lemma \ref{snip}, there exists a forest $F$ with $F\subseteq T$ such that $F$ has at least $\frac{3n}{2}-\frac{11}{2}$ edges and no paths of order $n$, so we may color all of the edges of $F$ red without creating a red $P_n$.  There are at most $\frac{5n}{2}-\frac{15}{2}-(\frac{3n}{2}-\frac{11}{2})=n-2$ edges remaining in $E(G)\setminus E(F)$, all of which we may color blue without creating a blue $P_n$.  
\end{proof}

\subsection{Main lemma}\label{sec:mainlem}

We will only use the following lemma in the case where $k=1$ or $k=2$, but we state it in general here.  Note that for instance when $k=1$, this says that if $G$ is a graph on $n-1<N\leq 2(n-1)$ vertices, then there is a bipartition of $V(G)$ into sets of order $n-1$ and $N-(n-1)$ such that the number of edges crossing the partition is approximately what we would get by taking a random such partition of a graph with $|E(G)|-N$ edges.

\begin{lemma}\label{partition}
There exists $n_0$ such that for all $n\geq n_0$, if $G=(V,E)$ is a graph on $N\geq n$ vertices, then the following holds.  Let $k$ be a positive integer uniquely defined by $k(n-1)<N\leq (k+1)(n-1)$ where $k\leq n^{1/32}$ and let $\alpha=\frac{n-1}{N}$.  If every component of $G$ has at least $n$ vertices, $\Delta(G)\leq N^{1/16}$, and $|E|\leq 100N\leq 100(k+1)n$, then there exists a partition of $V$ into $k+1$ parts $V_1, \dots, V_{k+1}$ such that $|V_1|,\dots,|V_k|, |V_{k+1}|\leq n-1$ and $|V_{k+1}|\leq N-k(n-1)+N^{15/16}$ and the number of edges crossing the partition is at most $(1-k\alpha^2-(1-k\alpha)^2)(|E|-N)+N^{15/16}$.
%=2(\binom{k}{2}\alpha_1^2+k\alpha_1\alpha_2)m.
\end{lemma}

The first tool needed to prove Lemma \ref{partition} is the following fact mentioned by Alon \cite{A}, stated in general and made explicit here.

\begin{lemma}\label{treepartition}
Let $G$ be a connected graph on $p$ vertices with maximum degree $\Delta$. For any $1\le \ell < p$, we can find a collection of connected subgraphs $S_1, \ldots, S_t$ of $G$ such that
\begin{enumerate}[label=(T\arabic*)]
\item\label{t1} $V(S_1), \ldots, V(S_t)$ form a partition of $V(G)$ with $\ell < |S_i| \le 1+\Delta \ell $ for all $i\in [t-1]$ and $|S_t| \le 1 + \Delta \ell$
%\item $S_i$ is connected for all $i\in [t]$.
\item\label{t2} $\sum_{i=1}^t |E(S_i)| \ge p - t$
\item\label{t3} if $\ell  = \floor{\sqrt{p}},$ then $\frac{1}{\Delta+1} \sqrt{p} \le t \le \sqrt{p}+1$
\end{enumerate} 
\end{lemma}

\begin{proof}
Let $T_0$ be a rooted spanning tree of $G$ with (arbitrary) root $r$.  For a rooted tree $T$ and vertex $v$, let $s(T,v)$ denote the subtree of $T$ rooted at vertex $v$ and let $C(v)$ denote the set of children of $v$.  Assume $T_i$ has been defined for some $i\ge 0$ and that $r$ is still the root of $T_i$. Traverse down $T_i$ from $r$ until encountering a vertex $v$ (if one exists) such that $|s(T_i,v)| > \ell$ and $|s(T_i,u)| \le \ell$ for all  $u\in C(v)$. Then $s(T_i,v)$ satisfies
\begin{align}\label{eq:t-right-size}
\ell < |s(T_i,v)| = 1 + \sum_{u\in C(v)}|s(T_i,u)| \le 1 + \Delta \ell.
\end{align}
If $v\neq r$, let $S_{i+1} = s(T_i, v)$ and $T_{i+1} = T_i - S_{i+1}$ and repeat for $i+1$. If $v=r$ or if no such vertex $v$ exists, then set $S_{i+1}=S_t = T_i$. 
Each $S_i$ is connected by construction. Property \ref{t1} is satisfied by \eqref{eq:t-right-size}.  Property \ref{t2} follows since each $S_i$ is connected and thus  $\sum_{i=1}^t |E(G[S_i])|  \geq  \sum_{i=1}^t (|S_i| -1) = p - t$.  
%Of course we can actually include all edges in $G[V(S_i)]$ if desired.

Finally, if $\ell= \floor{\sqrt{p}}$ we have 
\[     (t-1)(\floor{\sqrt{p}} + 1) \le  \sum_{i=1}^t|S_i| = p \le t(1+ \Delta\sqrt{p})\]
and from each of $(t-1)(\floor{\sqrt{p}} + 1) \le p$ and $p \le t(1+ \Delta\sqrt{p})\leq t\sqrt{p}(1+\Delta)$, we derive the bounds on $t$ in \ref{t3}.
\end{proof}

%\begin{corollary}\label{cor:sqrtp}
%Let $G$ is a connected graph on $p$ vertices with $\Delta(G)\le \Delta$ and let $k= \floor{\sqrt{p}}$. Then there exist connected subgraphs $S_1,\ldots, S_t$ satisfying \ref{t1} and \ref{t2} where $\frac{1}{\Delta+1} \sqrt{p} \le t \le \sqrt{p}+1.$
%\end{corollary}
%\begin{proof}
%We have 
%\[     (t-1)(\floor{\sqrt{p}} + 1) \le  \sum_{i=1}^t|S_i| = p \le t(1+ \Delta\sqrt{p})\]
%and from each of $(t-1)(\floor{\sqrt{p}} + 1) \le p$ and $p \le t(1+ \Delta\sqrt{p})$, the bounds on $t$ follow.
%\end{proof}

The next tool we need is the following concentration inequality of McDiarmid \cite{M} (see also \cite{FK}).  We note that McDiarmid's inequality isn't strictly necessary in this first application, but as we will use it again later in a different context, it is easiest to use it for both purposes.    
%We note that it would suffice to use Chernoff's inequality here, but as we will use McDiarmid's inequality again later, it is simpler to just use it for both purposes.

\begin{lemma}[McDiarmid's inequality]\label{lem:mcd} Let $Z=Z(X_1, \ldots, X_N)$ be a random variable that depends on $N$ independent random variables $X_1, \ldots, X_N$.  Suppose that 
\[|Z(X_1,\ldots, X_k, \ldots, X_N ) - Z(X_1, \ldots, X_k', \ldots, X_N)| \le c_k\] 
for all $k=1,\ldots, N$ and $X_1,\ldots, X_n, X'_k$. Then for any $t\ge 0$ we have
\[\mathbb{P}\sqbs{Z\ge \mathbb{E}[Z]  + t} \le \exp\of{-\frac{t^2}{2\sum_{k\in [N]} c_k^2 }}.
\]
\end{lemma}

We are now ready to prove the main lemma.

\begin{proof}[Proof of Lemma \ref{partition}]
Apply Lemma \ref{treepartition} with $\ell = \floor{\sqrt{N}}$ to partition the components of $G$ into  $\frac{\sqrt{N}}{\Delta+1} \le t\le \sqrt{N}+1$ connected subgraphs $S_1, \ldots, S_t$ each of order at most $1+\Delta\sqrt{N}$.  There are at least $N-(t-1)\geq N-\sqrt{N}$ edges accounted for in these subgraphs.  Define $m=|E|-(N-\sqrt{N})$ to be an upper bound on the number of edges of $G$ which are not contained in these subgraphs.  

We independently at random place each such connected subgraph in one of the sets $V_1, \dots, V_k, V_{k+1}$ with probability $\alpha$ for all $V_i$ with $i\in [k]$ and probability $1-k\alpha$ for $V_k$. Let $Z_i$ represent the number of vertices which land in the set $V_i$ for all $i\in [k+1]$.

Then $\mathbb{E}\sqbs{Z_1} = \dots= \mathbb{E}\sqbs{Z_k} = \alpha N=n-1$ and $\mathbb{E}\sqbs{Z_{k+1}} = (1-k\alpha)N=N-k(n-1)$.
Note that changing the position of one of $S_1, \ldots, S_t$ can change any of these variables by at most $1+\Delta\sqrt{N} \le 1+N^{9/16}$.
Thus we may apply McDiarmid's inequality (Lemma \ref{lem:mcd}) and the union bound to conclude that the probability that for some $i\in [k]$, $Z_i$ exceeds $n-1+N^{7/8}$ or $Z_{k+1}$ exceeds $N-k(n-1) + N^{7/8}$ is at most
\[(k+1)\cdot\exp\of{- \frac12\cdot\frac{N^{7/4}}{(\sqrt{N}+1)\cdot(1+N^{9/16})^2}   } = \exp\of{-\Omega(N^{1/8})}.\]
Thus at least $1 - e^{-\Omega(N^{1/8})}$ proportion of the partitions satisfy 
\begin{align}\label{eq:parts-small}
|V_1|, \dots, |V_k| \le n-1+N^{7/8}\,\, \textrm{ and }\,\, |V_{k+1}|\le N-k(n-1) + N^{7/8}.
\end{align}  
Now, by linearity of expectation, the expected number of edges $\mu$ crossing the partition satisfies
\begin{align*}
\mu\le (1-k\alpha^2-(1-k\alpha)^2)m.
\end{align*}
So there is a partition $V_1, \dots, V_{k}, V_{k+1}$ satisfying \eqref{eq:parts-small} with at most $(1-k\alpha^2-(1-k\alpha)^2)m+1$ edges crossing the partition; otherwise we would have
\[ (1-k\alpha^2-(1-k\alpha)^2)m\ge \mu \ge (1 - e^{-\Omega(N^{1/8})})((1-k\alpha^2-(1-k\alpha)^2)m+1)>(1-k\alpha^2-(1-k\alpha)^2)m,\]
a contradiction.

Finally, in order to achieve the desired upper bounds on the sizes of $V_1, \dots, V_k, V_{k+1}$, we potentially have to slightly modify the partition given above.  When modifying the partition, we only want to move vertices which have bounded degree, so let $S=\{v\in V(G): d(v)\leq 400k\}$ and note that 
\begin{equation}\label{Sbig}
|S|>(1-\frac{1}{2k})N;
\end{equation}
as otherwise there are at least $\frac{N}{2k}$ vertices of degree greater than $400k$ which gives $$200N\geq 2|E| = \sum_{v\in V(G)} d(v)> \frac{N}{2k}\cdot 400k=200N,$$ a contradiction.  

Now if $|V_i|>n-1$ for $i\in [k+1]$, there must exist $j\in [k+1]\setminus \{i\}$ such that $|V_j|<n-1$, so we select a vertex from $V_i\cap S$ and we move it to $V_j$.  Because $|S|>(1-\frac{1}{2k})N$ and by \eqref{eq:parts-small}, we can repeat this process for at most $kN^{7/8}$ steps until we have $|V_1|, \dots, |V_k|, |V_{k+1}|\leq n-1$ and $|V_{k+1}|\leq N-k(n-1)+(k+1)N^{7/8}\leq N-k(n-1)+N^{15/16}$.  At the end of this process, the number of edges crossing the partition is at most $$(1-k\alpha^2-(1-k\alpha)^2)m +1+ kN^{7/8}\cdot 400k < (1-k\alpha^2-(1-k\alpha)^2)(|E|-N) + N^{15/16}$$
as desired
%
%If $|V_{k+1}|<\frac{N}{2k^3}$, we move vertices from $S\cap (V_1\cup \dots \cup V_k)$ to $V_{k+1}$ until $|V_1|, \dots, |V_k|\leq  n-1$; a total of at most $kN^{7/8}$ vertices by \eqref{eq:parts-small}.  If $|V_{k+1}|\geq \frac{N}{2k^3}$, we do the following: if $|V_i|>n-1$ for $i\in [k]$ or $|V_{k+1}|>N-k(n-1)$, there must exist $j\in [k]$ such that $|V_j|<n-1$ or $|V_{k+1}|<N-k(n-1)$, so we select a vertex from $V_i\cap S$ and we move it to $V_j$.  Because of the size of $|S|$, we can repeat this process until we have $|V_1|=\dots=|V_k|=n-1$ and $|V_{k+1}|=N-k(n-1)$.  The total number of vertices moved will be at most $kN^{7/8}$.  In either case, at the end of this process, the number of edges crossing the partition is at most $$(1-k\alpha^2-(1-k\alpha)^2)m +1+ kN^{7/8}\cdot 800k^3 < (1-k\alpha^2-(1-k\alpha)^2)(|E|-N) + N^{15/16}.$$
\end{proof}

\subsection{Extending Proposition \ref{3n/2}}\label{extend}

The following observations extend Proposition \ref{3n/2}.  We note that there is a similarity between this observation and the concept of the \emph{integrity} of a graph (see \cite{V}).

\begin{observation}\label{n/2-path}
If $G$ has a set $S$ of at most $\frac{n}{2}-1$ vertices such that every component of $G-S$ has no path of order $n$, then there exists a 2-coloring of the edges of $G$ such that every monochromatic path has order less than $n$.
\end{observation}

\begin{figure}[ht]
\begin{center}
\includegraphics[scale=1]{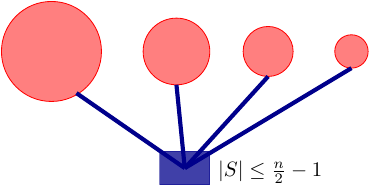}
\caption{Coloring the edges of $G$ in Observation \ref{n/2-path}}\label{obs31}
\end{center}
\end{figure}

\begin{proof}
We color all edges incident to $S$ with blue and every other edge red. By the condition on the components of $G-S$, there will clearly be no red path of order $n$. Any pair of consecutive vertices on a blue path must contain at least one vertex of $S$. Thus the longest blue path is of order less than $n$.
\end{proof}

We also note that there is a similarity between the following observation and the concept of the \emph{edge integrity} of a graph (see \cite{AHK}).

\begin{observation}\label{edgeint}
If $G$ has a subgraph $H$ such that $H$ has no path of order $n$ (in particular, if $H$ has at most $n-2$ edges) and every component of $(V(G), E(G)\setminus E(H))$ has order less than $n$, then there exists a 2-coloring of the edges of $G$ such that every monochromatic path has order less than $n$.
\end{observation}

\begin{proof}
Color the edges of $H$ with red and color the remaining edges blue.
\end{proof}

The following lemma says that if the number of vertices is not too much more than $3n/2$ and the number of edges of $G$ is small enough, we can essentially color $G$ in a way which resembles the coloring in Proposition \ref{3n/2}.

Given a graph $H$, and a positive real $\rho>0$, we say that a graph $G=(V,E)$ is \emph{$(\rho,H)$-free} if the graph $G'$ obtained from $G$ by adding a set of vertices $U$ with $|U|\leq \rho$ such that for all $u\in U$, $N(u)=(U\cup V)\setminus\{u\}$ does not contain a copy of $H$.  For example, every graph with at most $n-2\ceiling{\rho}-2$ edges is $(\rho, P_n)$-free.

\begin{lemma}\label{close3n/2}
Let $0<\ep<\frac{1}{100}$, let $n$ be sufficiently large, and let $G=(V, E)$ be a graph with $\delta(G)\geq 3$ and  $(\frac{3}{2}-\ep)n<|V|<(\frac{5}{3}-2\ep)n$.  Let $0<\sigma\leq \frac{1}{6}-\ep$ be defined by $|V|=(\frac{3}{2}-\ep+\sigma)n$ and let $d=\min\{\floor{\frac{1/2-3\ep}{\sigma}}+1, 100\}$.  
If $$|E|\leq \left(\frac{3(d+1)+6\sigma}{4}-d\ep\right)n,$$
then there exists a 2-coloring of $G$ such that every blue component has order at most $n-1$ and the graph $G_R$ induced by the red edges is $(\frac{\ep n}{2}, P_{n})$-free.
\end{lemma}

\begin{proof}
First note that $d$ is an integer with $d> \frac{1/2-3\ep}{\sigma}\geq \frac{1/2-3\ep}{1/6-\ep}=3$ and thus $d\geq 4$. Note that since $\sigma$ depends on $|V|$ and $d$ depends on $\sigma$, the upper bound on $|E|$ is a piecewise function of $|V|$ which increases on every interval corresponding to a fixed value of $d$, but which decreases as $d$ increases (roughly, the upper bound decreases as $|V|$ increases - see Figure \ref{3.75}).

The proof is based on the following claim.
\begin{claim}\label{XYZpartition}
There exists a partition $\{X,Y,Z\}$ of $V$ such that
\begin{enumerate}
\item every vertex in $X$ has at most one neighbor in $Z$ and
\item $|Z|\leq n-1$, $|Y|\leq (1/2-\ep)n$, and $|X|+|Y|\leq n-1$.
\end{enumerate}
\end{claim}

Before proving the claim, note that if such a partition exists, we can color all edges inside $Z$ and all edges inside $X\cup Y$ blue and all edges in $[X\cup Y, Z]$ red so that clearly every blue component has at most $n-1$ vertices (see Figure \ref{fig:XYZ}).  To see that the red graph $G_R$ is $(\ep n /2, P_{n})$-free, note that if $V_0$ is a set of at most $\ep n/2$ vertices adjacent to everything in $G_R$, then any red path can use at most $2|V_0|+2$ vertices from $X$, all vertices from $|V_0\cup Y|$ and $\min\{|Z|, |Y|+1\}$ vertices from $Z$ and $$2|V_0|+2+|V_0|+|Y|+|Y|+1=2|Y|+3|V_0|+3\leq n-1.$$

\begin{figure}[ht]
\centering
\subfloat[Subfigure 1 list of figures text][Finding the desired partition if $Y^*$ is small enough]{
\includegraphics[scale=1]{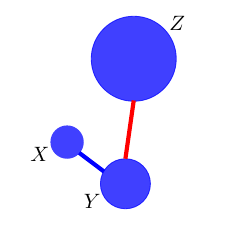}
\label{proof2i}}
\hspace{1in}
\subfloat[Subfigure 1 list of figures text][Finding the desired partition after moving vertices from  $Y^*$]{
\includegraphics[scale=1]{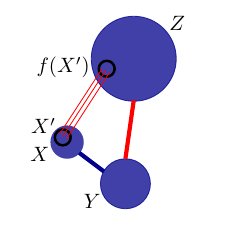}
\label{proof2ii}}
%\hfill
%\subfloat[Subfigure 1 list of figures text][Coloring the edges of $G$]{
%\includegraphics[scale=1]{fig_proof3i.pdf}
%\label{proof2iii}}
\caption{Coloring the edges in Lemma \ref{close3n/2}}\label{fig:XYZ}
\end{figure}

\begin{proof}[Proof of Claim \ref{XYZpartition}]
Let $X^*=\{v\in V: d(v)\leq d\}$.  We first show that $|X^*|$ is significantly larger than $\sigma n$.  Indeed, using $\delta(G)\geq 3$ we have 
\begin{align*}
\left(\frac{3(d+1)+6\sigma}{2}-2d\ep\right)n\geq 2|E|\ge \sum_{v\in V}d(v)&\geq 3|X^*|+(d+1)(|V|-|X^*|)\\
&=(d+1)|V|-(d-2)|X^*|.
\end{align*}

Rearranging and using $|V|=(3/2-\ep+\sigma) n$ gives 
\begin{align*}
(d-2)|X^*|&\geq (d+1)(3/2-\ep+\sigma) n-\left(\frac{3(d+1)+6\sigma}{2}-2d\ep\right)n= (d-2)\sigma n+(d-1)\ep n,
\end{align*}
and thus
\begin{equation}\label{X*}
|X^*|> (\sigma +\ep)n.
\end{equation}

Now let $X\subseteq X^*$ such that $|X|=\floor{(\sigma + \frac{\ep}{d-1})n}$, let $Y^*=N(X)\setminus X$, and note that
\begin{equation}\label{Y*}
|Y^*| \le d|X|.
\end{equation}

Since $N(X)\subseteq X\cup Y^*$ we would be done if $|Y^*| \le (1/2-2\ep)$ by letting $Y\subseteq V\setminus X$ such that $Y*\subseteq Y$ and $|Y|=\floor{(1/2-\ep)n}$ and letting $Z=V\setminus (X\cup Y)$ and noting that 
\begin{align*}
|Z|=|V|-|X|-|Y|&= (3/2-\ep+\sigma)n-|X|-|Y|\\
&\stackrel{\eqref{X*}}{\leq} (3/2-\ep+\sigma)n-(\sigma +\frac{\ep}{d-1})n-(1/2-\ep)n \leq n-1.  
\end{align*}
where the last inequality holds since $n$ is sufficiently large.

We now show that if $|Y^*| > (1/2-2\ep)$, then we can move at least $|Y^*|-(1/2-2\ep)$ vertices from $Y^*$ to $Z$.  We do this by showing that there exists an induced matching in the bipartite graph $G[X, Y^*]$ of size at least $|Y^*|-(1/2-2\ep)$.  

Let $Y_1=\{v\in Y^*: d(v, X)=1\}$ and $Y_2=\{v\in Y^*: d(v, X)\geq 2\}$.  We note that since every vertex in $X$ sends at most $d$ edges to $Y^*$, $G[X, Y^*]$ has an induced matching of size at least $|Y_1|/d$.  We have 
\begin{align*}
d|X|\geq e(X, Y^*)\geq |Y_1|+2(|Y^*|-|Y_1|)=2|Y^*|-|Y_1|
\end{align*}
which implies 
\begin{align*}\frac{|Y_1|}{d} \geq \frac{2|Y^*|}{d}-|X|= |Y^*| - \frac{d-2}{d}|Y^*| - |X|
&\stackrel{\eqref{Y*}}{\ge} |Y^*| - (d-1)|X|\\
&= |Y^*|-(d-1)\floor{(\sigma +\frac{\ep}{d-1})n}\\
&\geq |Y^*|-(d-1)\sigma n-\ep n\\
&\geq |Y^*|-\frac{1/2-3\ep}{\sigma}\sigma n-\ep n\\
&=|Y^*|-(\frac{1}{2}-2\ep)n
\end{align*} 
where the last inequality holds by the upper bound on $d$.

Let $M$ be a maximum induced matching in $G[X,Y^*]$ where $X'$ are the vertices from $X$ which are saturated by $M$ and $f(X')$ are the vertices in $Y^*$ which are saturated by $M$.  Set $Y'=Y^*\setminus f(X')$.  By the above we have $|Y'|\leq (1/2-2\ep)$.

Finally, let $Y\subseteq V\setminus X$ such that $Y'\subseteq Y$ and $|Y|=\floor{(1/2-\ep)n}$.  Now let $Z=V\setminus (X\cup Y)$ and note that  as before
\begin{align*}
|Z|=|V|-|X|-|Y|&= (3/2-\ep+\sigma)n-|X|-|Y|\\
&\stackrel{\eqref{X*}}{\leq} (3/2-\ep+\sigma)n-(\sigma +\frac{\ep}{d-1})n-(1/2-\ep)n \leq n-1.  
\end{align*}
where the last inequality holds since $n$ is sufficiently large.

This completes the proof of Claim \ref{XYZpartition}.
\end{proof}
Having established Claim \ref{XYZpartition}, we have completed the proof of Lemma \ref{close3n/2}.
\end{proof}

\section{Two colors}\label{sec:2-col}

We are now ready to give a proof of Theorem \ref{thm:main-2-col}.   
We note that the $\ep$ in the following proof can be taken to be as small as $\ep =n^{-\Theta(1)}$; however, for the sake of readability, we didn't try to optimize the value of $\ep$.

\begin{proof}[Proof of Theorem \ref{thm:main-2-col}]
Let $0<\ep<1/100$ and let $n_0$ be a sufficiently large integer (the value of which we don't explicitly compute, but we will point out which inequalities depend on $n$ being sufficiently large).
Let $G'=(V',E')$ be a connected graph with at most $(3+\gamma-\ep)n$ edges, where $0\leq \gamma< 3/2+\ep$ is to be chosen later (ultimately, we will choose $\gamma=3/4$ and in fact our proof can handle larger values of $\gamma$, but restricting $\gamma$ in this way makes it easier to apply Lemma \ref{partition} later in the proof).  By Lemma \ref{mindegree} it suffices to assume that $\delta(G')\geq 3$ and thus $|V'|\leq 2|E'|/3\leq (2+2\gamma/3-2\ep/3)n$.  We will exhibit a 2-coloring of $G'$ with no monochromatic $P_n$, but since we are using Lemma \ref{mindegree} this will prove that all graphs with at most $|E'|$ edges will have a 2-coloring with no monochromatic $P_{n+2}$ (in other words, we will ultimately be showing that $\hat{R}(P_{n+2})\leq (3.75-o(1))(n+2)=(3.75-o(1))n$).  So by Proposition \ref{3n/2} we may assume that $|V'|\ge \frac{3}{2}n-\frac{3}{2}$.

Let $V_0=\{v\in V(G'): d(v)>n^{1/32}\}$.  We have 
$n^{1/32}|V_0|\leq 2|E'|$ and thus since $n$ is sufficiently large, 
\begin{equation}\label{V0}
|V_0|\leq 2(3+\gamma-\ep)n^{31/32}\leq \frac{\ep^2 n}{2}.  
\end{equation}
We say that a component $C$ of $G'-V_0$ is \emph{small} if $|C|<n$ and \emph{large} otherwise.

Suppose there are exactly $t$ edges between $V_0$ and the large components of $G'-V_0$.  Now let $G=(V,E)$ be the graph obtained from $G'$ by deleting all of the vertices in $V_0$, deleting all of the vertices in small components of $G'-V_0$, and adding $t$ edges inside $V$ so that the minimum degree of $G$ is at least 3.  So $|E|\leq |E'|$.

We claim that we can 2-color the edges of $G$ in such a way that the graph $G_B$ induced by the blue edges has no components of order at least $n$ and the graph $G_R$ induced by the red edges is $(\frac{\ep^2 n}{2}, P_{n})$-free.  Once we establish this claim, we complete the proof of Theorem \ref{thm:main-2-col} by coloring all edges between $V_0$ and the large components of $G'-V_0$ red, the edges inside $V_0$ blue, the edges between $V_0$ and the small components blue, and the edges inside the small components red.

\begin{figure}[ht]
\centering
\subfloat[Subfigure 1 list of figures text][Coloring the edges of $G$ in Case 1]{
\includegraphics[scale=1]{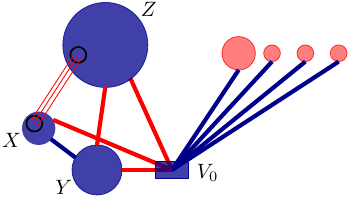}
\label{proof1}}
\hfill
\subfloat[Subfigure 1 list of figures text][Coloring the edges of $G$ in Case 2 and Case 3]{
\includegraphics[scale=1]{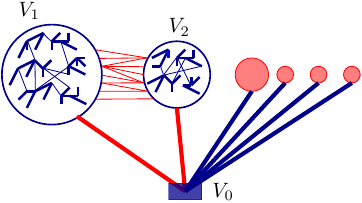}
\label{proof23}}
\caption{Coloring the edges in the proof of Theorem \ref{thm:main-2-col}}
\end{figure}

Note that by Proposition \ref{3n/2}, we may assume $|V|> (3/2-\ep^2)n$ since any graph with at most $(3/2-\ep^2)n$ vertices can be 2-colored in such a way that the red graph $G_R$ is $(\frac{\ep^2n}{2}, P_n)$-free and every component of the blue graph $G_B$ has order at most $n-1$.

\noindent
\textbf{Case 1} ($(\frac{3}{2}-\ep^2)n< |V|< (\frac{5}{3}-2\ep^2)n$) In this case we are done by applying Lemma \ref{close3n/2} (with $\ep^2$ in place of $\ep$).

\noindent
\textbf{Case 2} ($(\frac{5}{3}-2\ep^2)n\leq |V|\leq 2(n-1)$)  We parameterize this by introducing a variable $\sigma$ such that $|V|=(3/2-\ep^2+\sigma)n$ and thus $1/6-\ep^2\leq \sigma< 1/2+\ep^2$.

Lemma \ref{partition} (with $k=1$) provides a bipartition of $V$ into sets $V_1, V_2$ of order at most $n-1$ such that the number of edges crossing the partition is at most
\begin{equation*}
\left(1-\left(\frac{1}{1+(1/2+\sigma)}\right)^2-\left(\frac{1/2+\sigma}{1+(1/2+\sigma)}\right)^2+2\ep^2\right)(3/2+
\gamma-\sigma-\ep)n+|V|^{15/16}<(1-\frac{\ep}{4})n,
\end{equation*}
where the last inequality holds provided $n$ is sufficiently large and 
\begin{equation}\label{gammaupper}
\gamma\leq \frac{3/4+\sigma+3\sigma^2}{1+2\sigma}.
\end{equation}
Since we are assuming $1/6-\ep^2\leq  \sigma< 1/2+\ep^2$, we have $\frac{3/4+\sigma+3\sigma^2}{1+2\sigma}\geq 0.75-\ep$ with the minimum occurring when $\sigma= \frac{1}{6}-\ep^2$.
%(note that the minimum of $\frac{3/4+\sigma+3\sigma^2}{1+2\sigma}$ over the entire interval $0< \sigma< \frac{1}{2}+\ep^2$ is $\sqrt{3}-1\approx 0.732$ and occurs when $\sigma=\frac{2\sqrt{3}-3}{6}\approx 0.07735$).

Now color the edges inside the sets $V_1, V_2$ blue and the edges between the sets $V_1$, $V_2$ red.  Since the red graph $G_R$ has at most $(1-\frac{\ep}{4})n$ edges, $G_R$ is $(\frac{\ep^2 n}{2}, P_{n})$-free.

\noindent
\textbf{Case 3} ($2(n-1)< |V|\leq 3(n-1)$) We parameterize this by introducing a variable $\tau$ and assuming that $|V| = (2+\tau)(n-1)$ where $0 < \tau \leq 1$.
Apply Lemma \ref{partition} (with $k=2$) to get a tripartition of $V$ into sets $V_1, V_2, V_3$ of order at most $n-1$ such that the number of edges crossing the partition is at most
\begin{align*}
\left(1-2\left(\frac{1}{2+\tau}\right)^2-\left(\frac{\tau}{2+\tau}\right)^2\right)(1+\gamma-\tau-\ep)n+|V|^{15/16}<(1-\frac{\ep}{4})n,
\end{align*}
where the last inequality holds provided $n$ is sufficiently large and 
\begin{equation*}
\gamma\leq \frac{1+\tau+5\tau^2/2}{1+2\tau}.
\end{equation*}
We have $\frac{1+\tau+5\tau^2/2}{1+2\tau}\geq \frac{3}{4}(\sqrt{5}-1)\approx 0.927$
with the minimum occurring when $\tau=\frac{3\sqrt{5}-5}{10}\approx 0.1708$.

Now color the edges inside the sets $V_1, V_2, V_3$ blue and the edges between the sets $V_1$, $V_2$, $V_3$ red.  Since the red graph $G_R$ has at most $(1-\frac{\ep}{4})n$ edges, $G_R$ is $(\frac{\ep^2 n}{2}, P_{n})$-free.
\end{proof}

\begin{figure}[ht!]
\begin{center}
\includegraphics[scale=6]{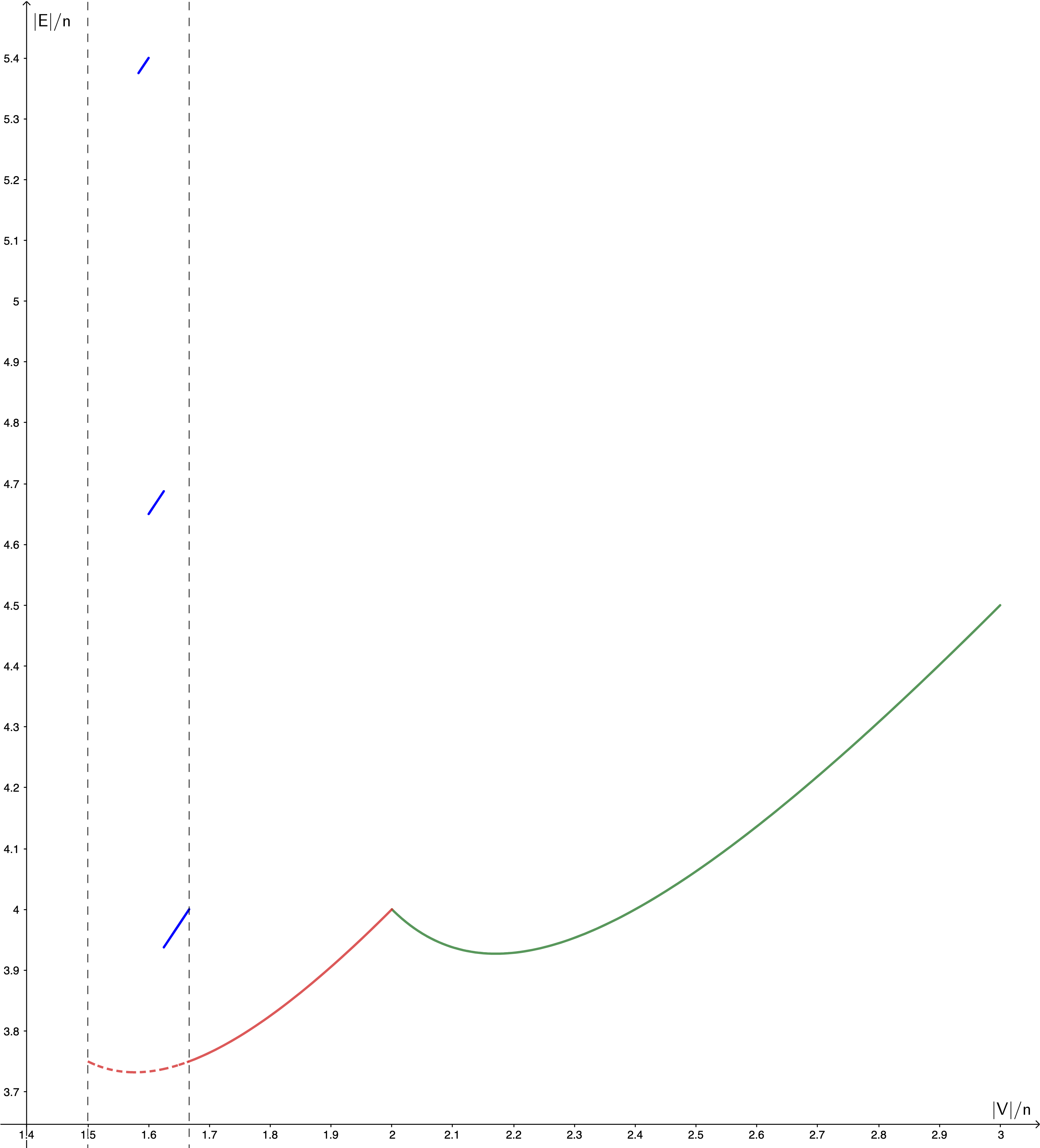}
\caption{The value of $c$ on the $x$-axis represents $|V|=cn$.  For a given value of $c$, the curve shows the maximum number of edges $G$ can have so that our proof gives a 2-coloring of $G$ with no monochromatic $P_n$.  The blue curve corresponds to Case 1, the solid red curve to Case 2, and the green curve to Case 3.  Note that the minimum over the entire interval is $3.75$ and occurs when $c=5/3$.}\label{3.75}
\end{center}
\end{figure}

One note about the previous proof.  If we were to deal with the case $(\frac{3}{2}-\ep^2)n<|V|\leq 2(n-1)$ by simply applying Lemma \ref{partition}, the bound we obtain in \eqref{gammaupper} would be $\gamma\leq \sqrt{3}-1$ which gives an overall lower bound of $\hat{R}(P_n)\geq (2+\sqrt{3}-o(1))n\approx 3.732n$ (this hypothetical scenario is depicted by the dashed red curve in Figure \ref{3.75}).  So by using Lemma \ref{close3n/2} to deal with the case $(\frac{3}{2}-\ep^2)n<|V|<(\frac{5}{3}-2\ep^2)n$ separately, we get an improvement of about $0.018n$.  In light of this, if one were to attempt to improve the lower bound of $(3.75-o(1))n$, a good test case would be when $|V|\approx\frac{5n}{3}$, since this corresponds to the case where $|V|=(\frac{3}{2}-\ep^2+\sigma)n$ and $\sigma\approx 1/6$ which is the bottleneck of the above proof. 

%We originally dealt with the case $|V|=(\frac{3}{2}n-\ep^2+\sigma) n$ as a whole rather than splitting into the subcases $0< \sigma\lesssim 1/6$ and $1/6\lesssim \sigma\lesssim 1/2$.  Without the subcases, the bound we obtained in \eqref{gammaupper} was $\gamma\leq \sqrt{3}-1$ which gives an overall lower bound of $\hat{R}(P_n)\geq (2+\sqrt{3}-o(1))n\approx 3.732n$.  So by dealing with the subcases separately, we got an improvement of about $(0.018-o(1))n$.  

Finally we note that by our result, one immediately obtains $$\hat{R}(C_n, C_n)\ge \hat{R}(P_n, C_n)\ge \hat{R}(P_n, P_n) \ge (3.75 - o(1))n$$ for all sufficiently large $n$. However when $n$ is odd, any graph $G$ with $|V(G)|\le 2n-2$ can be 2-colored in a way which avoids a monochromatic $P_n$ and $C_n$ (partition the vertices into 2 sets of size at most $n-1$ and note that the graph between the two sets contains no odd cycles).  Thus, we can use the bounds provided by Case 3 of the proof of Theorem \ref{thm:main-2-col} to obtain the following.
\begin{remark}
For all $\ep>0$ and sufficiently large $n$, if $G=(V, E)$ is a graph with $|V|\geq 2n-1$ and $|E|\leq (3+\frac{3(\sqrt{5}-1)}{4}-\ep)n$, then there exists a 2-coloring of the edges of $G$ such that every monochromatic path has order less than $n$.  Consequently, for all sufficiently large odd $n$, $\hat{R}(C_n, C_n)\geq \hat{R}(P_n, C_n)\geq (3+\frac{3(\sqrt{5}-1)}{4}-o(1))n\geq 3.927n$.
\end{remark}

The above improves a bound of $\hat{R}(P_n, C_n)\geq 3(n-1)$ (for odd $n$) noted by Dudek, Khoeini, and Pra{\l}at in \cite{DKP} and can be used to slightly improve a bound of $\hat{R}_r(C_n)\geq (3.75-o(1))2^{r-2}n$ (for odd $n$) given by Javadi and Miralaei in \cite{JM}.

\section{More than two colors} \label{sec:r-col}

The following statement implies the first part of Theorem \ref{thm:main-r-col}.

\begin{proposition}\label{r=3}
For all $\ep> 0$, $r\geq 2$, and sufficiently large $n$, if $G$ is a graph with at most $(\frac{(r-1)r}{2}+\alff-\ep)n$ edges, then there exists an $r$-coloring of the edges of $G$ such that every monochromatic path has order less than $n$.
\end{proposition}
 
\begin{proof}
Let $G=(V,E)$ be a connected graph with at most $(\frac{(r-1)r}{2}+\alff-\ep)n$ edges.  For $r=2$, the result holds by Theorem \ref{thm:main-2-col}.  So let $r\geq 3$ and suppose the result is true in the case $r-1$.  If $N\leq (r-1)(n-2)$, then we are done by Proposition \ref{3n/2_r}; so suppose $N\geq (r-1)(n-2)+1$.  Let $T$ be a spanning tree of $G$ and apply Lemma \ref{snip} to get a forest $F$ with no paths of order $n$ and at least $(r-1)(n-2)-2r-2=(r-1)n-4r$ edges. Color the edges of the forest with color $r$.  The number of remaining edges is at most $(\frac{(r-1)r}{2}+\alff-\ep)n-(r-1)n+4r=(\frac{(r-2)(r-1)}{2}+\alff-\ep')n$ (where $\ep' = \ep-4r/n >0$ since $n$ is sufficiently large) and thus we may apply induction to color the remaining edges with the remaining $r-1$ colors.   
\end{proof}

\begin{definition}
An affine plane of order $q$ is a $q$-uniform hypergraph on $q^2$ vertices (called points), with $q(q+1)$ edges (called lines) such that each pair of vertices is contained in exactly one edge.  
\end{definition}

It is well known that an affine plane of order $q$ exists whenever $q$ is a prime power (and it is unknown whether there exists an affine plane of non-prime power order).  We collect two key properties of affine planes in the following proposition.

\begin{proposition}\label{affprop}
Let $q\geq 2$ be such that there exists an affine plane of order $q$.  There exists a $q+1$-coloring of the edges of $K_{q^2}$ such that
\begin{enumerate}
\item\label{ap1} every color class (called a parallel class) consists of a collection of $q$ vertex disjoint $K_q$'s, and
\item\label{ap2} every vertex $v$ is contained in exactly one $K_r$ of each color and the union of these $q+1$ $K_r$'s incident with $v$ is all of $V(K_{q^2})$.
\end{enumerate}
\end{proposition}

The following theorem implies the second part of Theorem \ref{thm:main-r-col}.  We modify Krivelevich's proof \cite[Theorem 8]{K} in such a way that no color is ``wasted'' on the high degree vertices.  This improves the lower bound from $((r-2)^2-o(1))n$ to $((r-1)^2-o(1))n$.

\begin{proposition}\label{affine}
Suppose that an affine plane of order $q$ exists and suppose $n$ is sufficiently large. For all graphs $G$ with at most $q^2 n - 6q^4n^{0.9}  = (q^2-o(1)) n$ edges, there exists a $q+1$-coloring of the edges of $G$ such that every monochromatic path has order less than $n$.
\end{proposition}

\begin{proof}
Let $G=(V,E)$ be a graph with $|E|\leq q^2n-6q^4n^{.9}$.  Let $V_0:=\set{v\in V(G)}{d(v) \ge n^{0.1}}$. Then $q^2 n \ge |E(G)| \ge \frac{1}{2}|V_0|n^{0.1}$ implies that 
$|V_0| \le 2q^2n^{0.9}$.
Now randomly partition $V\setminus V_0$ into $q^2$ parts $V_1, \ldots V_{q^2}$ by placing each vertex into one of these sets independently with probability $1/q^2$.
 Let $L$ be a line of the affine plane $A_{q}$ on point set $[q^2]$. For each edge $e$ in $G[V\setminus V_0]$, we assign color $i$ to $e$ if the endpoints of $e$ are in distinct sets $V_x$ and $V_y$ where the unique line containing $x$ and $y$ in $A_{q}$ is in the $i$'th parallel class of $A_{q}$. We color $e$ arbitrarily if both of its endpoints are in $V_x$ for some $x$.
 
For a line $L$ of $A_{q}$, define the random variable $X_L := |E\of{\bigcup_{x\in L}V_x}|$. 
Then \[\mathbb{E}\sqbs{X_L}\le \frac{1}{q^2}\cdot |E(G)| \leq n - 6q^2n^{0.9}.\]
 Since every vertex of $V\setminus V_0$ has degree at most $n^{0.1}$, we have that moving any one vertex from $V_x$ to $V_y$ can change $X_L$ by at most $n^{0.1}.$ Thus we may apply McDiarmid's inequality (Lemma \ref{lem:mcd}) with $c_k = n^{0.1}$ for all $k$ to conclude that
 \[\mathbb{P}\sqbs{X_L \ge n-5q^2n^{0.9}} \le \exp\of{-\frac{(q^2n^{0.9})^2}{2|V\setminus V_0|\cdot (n^{0.1})^2}} = \exp\of{-\Omega(n^{0.6})},\]
where we used $|V\setminus V_0|\leq |E|\leq q^2n$ in the last inequality. Thus taking a union bound over all $(q+1)q$ lines of $A_q$, we conclude that there exists a partition of $V\setminus V_0$ in which at most $n - 5q^2n^{0.9}$ edges lie inside $\bigcup_{x\in L}V_x$ for all lines $L$.  In other words, for all $L$ in $A_q$, the graph induced by $\bigcup_{x\in L}V_x$ is $(2q^2n^{0.9}, P_n)$-free. Suppose $V_1, \ldots, V_{q^2}$ is such a partition. 
 
Finally, we must color the edges incident with $V_0$.  We color the edges from $V_0$ to $V_0\cup V_1$ arbitrarily, and for all $i\in [q^2]\setminus\{1\}$ we color the edges from $V_0$ to $V_i$ the same as the color of the edges between $V_1$ and $V_i$.  By Proposition \ref{affprop}(ii), this accounts for all of the edges incident with $V_0$, and since for all $L$ containing $V_1$ the graph induced by $\bigcup_{x\in L}V_x$ is $(2q^2n^{0.9}, P_n)$-free, we have that the graph induced by $V_0\cup \bigcup_{x\in L}V_x$ is $P_n$-free.  
%
%Let $L_1, \dots, L_{q+1}$ be the lines from $A_q$ incident with the point $1$, one from each parallel class (which is possible ).  Note that for all $j\in [q^2]\setminus \{1\}$, $V_j$ intersects precisely one such $\bigcup_{x\in L_i}V_x$ for $i\in [q+1]$.  For each $i\in [q+1]$, we color the edges from $V_0$ to $\bigcup_{x\in L_i}V_x$ with color $i$ (coloring the edges from $V_0$ to $V_1$ arbitrarily).  Now every edge from $V_0$ to $V\setminus V_0$ has been colored and for each color $i\in [q+1]$, there exists a unique line $L_i$ such that $V_0$ sends edges of color $i$ to $\bigcup_{x\in L_i}V_x$.
%%Notice that every connected monochromatic subgraph of $G$ is contained in $\bigcup_{x\in L}V_x$ for some line $L$ of $A_{q}$ not containing the point 1, or is contained in $V_0\cup \bigcup_{x\in L}V_x$ for a line $L$ containing 1.  Suppose  $L$ is a line of $A_{q}$ containing point $1$. 
%Any path contained in  $V_0\cup \bigcup_{x\in L_i}V_x$ can have order at most \[\abs{E\of{\bigcup_{x\in L_i}V_x}} + 2|V_0| \le n - 5q^2n^{0.9} + 4q^2n^{0.9} < n\]
%The sets $\bigcup_{x\in L}V_x$ where $L$ does not contain point 1 still contain less than $n-1$ edges and thus have no path of order $n$.
\end{proof}

\section{Additional observations and conclusion}\label{sec:concl}

In this section we collect a few additional thoughts, none of which fit into into the main thread of the paper.  The four observations below quantify the intuitive notion that if $G$ is a graph having the property that every 2-coloring of the edges of $G$ contains a monochromatic $P_n$, then $G$ must be ``expansive'' in some sense. 

For a graph $G=(V,E)$, let $S_V$ be the set of permutations of $V$. The \emph{bandwidth}, $\varphi$ of $G$ is defined as \[\varphi(G):=\min_{f\in S_V}\max_{uv\in E}|f(u) - f(v)|.\]

\begin{observation}
For all graphs $G$, if $\varphi(G) \le \frac{n}{2}-1$, then there is a 2-coloring of the edges of $G$ such that every monochromatic path has order less than $n$.
\end{observation}

\begin{proof}
Choose an ordering $f$ of $V(G)$ which witnesses the bandwidth of $G$; i.e.\ $\max_{uv\in E}|f(u) - f(v)|=\phi(G)$. Now split the vertices into sets $V_1, \dots, V_t$, with $|V_1|=\dots=|V_{t-1}|=\floor{\frac{n}{2}-1}$ and $|V_t|\leq n-1$.  For all odd $i\in [t]$, color the edges from $V_i$ to $V_i\cup V_{i+1}$ red, and for all even $j\in [t]$ color the edges from $V_j$ to $V_j\cup V_{j+1}$ blue.
\end{proof}

A \emph{depth first search} (DFS) tree (or \emph{normal} tree) $T$ rooted at $x$ in a graph $G$ is a subtree of $G$ such that for all $uv\in E(G)$ with  $u,v \in V(T)$, either $v$ is on the $x-v$ path in $T$ or $u$ is on the $x-u$ path in $T$.  

For a connected subgraph $H$ of a graph $G$ and vertices $u,v\in V(H)$, let $d_H(u,v)$ be the length of the shortest path between $u$ and $v$ in $H$.  A \emph{breadth first search} (BFS) tree $T$ rooted at $x$ is a subtree of $G$ such that for all $v\in V(T)$,  $d_T(x, v)=d_G(x, v)$.  Such a tree has the property that for all $uv\in E(G)$ with $u,v\in V(T)$, $|d_T(x,u)-d_T(x,v)|\leq 1$.  The vertices at each fixed distance from the root are called the \emph{levels} of $T$.
It is well known that for every connected graph $G$ and every vertex $x\in V(G)$, there exists a spanning DFS tree $T$ rooted at $x$ and a spanning BFS tree rooted at $x$.  

Using the notation for rooted trees from the proof of Lemma \ref{treepartition}, we have the following observation.
\begin{observation}
Let $G$ be a connected graph.  If there exists a vertex $x$ and a DFS tree $T$ rooted at $x$ so that every child $y\in C(x)$ satisfies $|S(T,y)|\le \frac{5n}{4}-2$, then there exists a 2-coloring of the edges of $G$ such that every monochromatic path has order less than $n$.
\end{observation}

\begin{proof}
For each sub-tree $S(T,y)$ where $y\in C(x)$, we partition the vertices of $S(T,y)$ into sets $A_y$ and $B_y$ where $|A_y|\le \frac n4 -1$, $y\in A_y$ and $|B_y|\le n-1$. Let $A=\{x\}\cup \bigcup_{y\in C(x)}A_y$ and $B=\bigcup_{y\in C(x)}B_y$.  We color the edges of $G$ within $B$ blue and the edges from $A$ to $A\cup B$ red.  Note that this is all the edges of $G$ since no edges go between $S(T,y)$ and $S(T,z)$ for $y,z\in C(x)$, $y\neq z$. Clearly there are no blue paths of order $n$. Any red path may intersect at most two of the sub-trees $S(T,y)$, $S(T,z)$ for $y,z\in C(x)$, $y\neq z$ and any such path must pass through $x$. For all $y\in C(x)$, the longest possible red path in $G[A_y\cup B_y]$ is of order at most $\frac n2 -1$ and so the longest red path in $G$ is of order at most $n-1$.
\end{proof}

\begin{observation}
Let $G$ be a connected graph.  If there exists a vertex $x$ and a BFS tree $T$ rooted at $x$ such that every pair of consecutive levels of $T$  have fewer than $n$ vertices, then there exists a 2-coloring of the edges of $G$ such that every monochromatic path has order less than $n$.
\end{observation}

\begin{proof}
For all $i\geq 0$, let $D_i=\{v: d_T(x, v)=i\}$.  For all $j\geq 0$, color the edges from $D_{2j}$ to $D_{2j}\cup  D_{2j+1}$ red and the edges from $D_{2j+1}$ to  $D_{2j+1}\cup D_{2j+2}$ blue.  By the property of BFS trees, this accounts for every edge in $G$.  Since every two consecutive levels contain fewer than $n$ vertices, there are no monochromatic paths of order $n$.  
\end{proof}

The following observation was inspired by Figure 2 in both \cite{BKLL1} and \cite{BKLL2}.

\begin{observation}
If $G$ is a graph on $N$ vertices with $\alpha(G)\geq N-(n-3)$, then there exists a 2-coloring of the edges of $G$ such that every monochromatic path has order less than $n$.
\end{observation}

\begin{proof}
Let $S$ be an independent set of order at least $N-(n-3)$ and partition the vertices of $V(G)\setminus S$ into disjoint sets $X,Y$ with $|X|, |Y|\leq \frac{n}{2}-1$.  Color all edges incident with $X$ red and color all edges incident with $Y$ blue (so edges between $X$ and $Y$ can be either color).  The longest monochromatic path has order at most $2(\frac{n}{2}-1)+1=n-1$. 
\end{proof}

Finally, we end with the following question which relates to the upper bound on the size-Ramsey number of a path.

\begin{question}
What is the largest monochromatic path one can find in an arbitrary 2-coloring of a $d$-regular graph on $N$ vertices?
\end{question}

For instance, suppose it is always possible to 2-color the edges of 5-regular graph on $N$ vertices (with $N$ sufficiently large) so that there is no monochromatic path of order $\frac{N}{30}$.  This would imply that all 5-regular graphs on at most $30n$ vertices (which have at most $75n$ edges) have a 2-coloring with no monochromatic $P_n$; in other words, 5-regular graphs could never improve the current best \cite{DP2} upper bound $\hat{R}(P_n)\leq 74n$.

\bigskip

\noindent
\tbf{Acknowledgements:} We thank two very thorough referees for their careful reading of the paper and their helpful comments.

\end{document}